\documentclass[a4paper,12pt]{article}
\usepackage{pdfpages}
\usepackage{amsmath,amssymb,amsthm}
\usepackage{hyperref}
\hypersetup{colorlinks=true,urlcolor=blue,citecolor=blue,linkcolor=blue}
\usepackage{courier}
\usepackage{tikz}
\usetikzlibrary{calc,matrix,arrows,decorations.markings}
\usepackage{array}
\usepackage{color}
\usepackage{enumerate}
\usepackage{nicefrac}
\usepackage{listings}
\usepackage[utf8]{inputenc}
\usepackage[english]{babel}
\usepackage{blindtext}
\usepackage{amsthm}
\newtheorem{theorem}{Theorem}
\newtheorem{lemma}[theorem]{Lemma}
\newtheorem{conjecture}{Conjecture}

\title{Ramsey Numbers of Odd Cycles Versus Larger Even Wheels}

\author{Ryan Alweiss\thanks{MIT}}

\date{\today}

\begin{document}
\maketitle

\begin{abstract}

The generalized Ramsey number $R(G_1, G_2)$ is the smallest positive integer $N$ such that any red-blue coloring of the edges of the complete graph $K_N$ either contains a red copy of $G_1$ or a blue copy of $G_2$. Let $C_m$ denote a cycle of length $m$ and $W_n$ denote a wheel with $n+1$ vertices. In 2014, Zhang, Zhang and Chen determined many of the Ramsey numbers $R(C_{2k+1}, W_{n})$ of odd cycles versus larger wheels, leaving open the particular case where $n = 2j$ is even and $k<j<3k/2$. They conjectured that for these values of $j$ and $k$, $R(C_{2k+1}, W_{2j})=4j+1$. In 2015, Sanhueza-Matamala confirmed this conjecture asymptotically, showing that $R(C_{2k+1}, W_{2j}) \le 4j+334$. In this paper, we prove the conjecture of Zhang, Zhang and Chen for almost all of the remaining cases.  In particular, we prove that $R(C_{2k+1},W_{2j})=4j+1$ if $j-k \ge 251$, $k<j<3k/2$, and $j \ge 212299$.

\end{abstract}

\section{Introduction}

Given graphs $G_1$ and $G_2$, we define the \emph{Ramsey number} of $G_1$ and $G_2$ $R(G_1,G_2)$ to be the minimum positive integer $N$ such that a complete graph on $N$ vertices with edges colored red and blue contains a red $G_1$ or a blue $G_2$ as a subgraph.  These generalized Ramsey numbers have been well studied for various pairs of graphs $(G_1,G_2)$.  See \cite{rad} for a rather comprehensive survey on this topic.

First, we introduce some some notation and definitions.  We use $v_1 \ldots v_nv_1$ to denote the cycle $C_n$ with vertices $v_i$ where $1 \le i \le n$.  A \emph{wheel} $W_n$ is a graph on $n+1$ vertices $v_0, v_1, \ldots, v_n$ so that the induced subgraph on $v_i$ with $i>0$ is the cycle $C_n$, and also $v_0$ is adjacent to $v_i$ for each $1 \le i \le n$.  The vertex $v_0$ is called the \emph{hub} of the wheel.  For every graph $G$ and $v \in V(G)$, we use $N(v)$ to denote the set of vertices incident to $v$, and its elements are \emph{neighbors} of $v$.  A graph is \emph{connected} if there is a path between any two vertices.  A graph is \emph{2-connected} if it is connected and is not disconnected by the removal of any vertex.  A \emph{2-edge-colored graph} is a graph $H$ with each of its edges colored red or blue.  We use $H^R$ to denote the subgraph of $H$ formed by the red edges of $H$.  Similarly, $H^B$ denotes the subgraph of $H$ formed by the blue edges of $H$.  If $H$ is a set of vertices, then $H^R$ is defined to be $(G[H])^R$ and $H^B$ is defined to be $(G[H])^B$.  Given a $2$ coloring of $H$, we define $\delta^R(H)=\delta(H^R)$ and $\delta^B(H)=\delta(H^B)$.  Furthermore, the \emph{circumference} $c(H)$ is the length of the largest cycle of $H$, where we take the convention that $c(H)=\infty$ if the graph is acyclic.  If there exists a pair of integers $(x_1,x_2)$ with $3 \le x_1 \le x_2$ so that a graph $H$ has a cycle of length $x$ if and only if $x_1 \le x \le x_2$ then $H$ is \emph{weakly pancyclic}.  A weakly pancyclic graph $H$ containing a triangle and a cycle of length $|H|$ is \emph{pancyclic}.

It is natural to try to compute $R(C_m,W_n)$.  In 2014, Zhang, Zhang, and Chen showed that $R(C_m,W_n) \le \max \{ 2n+1, 3m-2 \}$ when $m \ge 5$ is odd and $n>m$ \cite{zhang}.  Furthermore, they proved $R(C_m,W_n)=2n+1$ for $m$ odd, $n \ge 3(m-1)/2$, except for $(m,n)=(3,3)$.  They also showed that $R(C_m,W_n)=3m-2$ when $m$ and $n$ are both odd and $m<n \le 3(m-1)/2$.  These two results show that the bound $\max \{ 2n+1, 3m-2 \} $ is optimal for all cases where $m \ge 5$ is odd and $n>m$, except the case where $m$ is odd, $n$ is even, and $m<n<3(m-1)/2$.  They conjectured that $R(C_m,W_n)=2n+1$ where $m$ is odd, $n$ is even, and $m<n<3(m-1)/2$.  This would show that the bound is optimal in this final case as well.  Letting $m=2k+1$ and $n=2j$, we have the main conjecture studied in this paper.

\begin{conjecture} \emph{\cite{zhang}}
For $k<j<3k/2$, $R(C_{2k+1},W_{2j})=4j+1$
\end{conjecture}

A graph with $4j$ vertices consisting of two disjoint blue copies of $K_{2j}$ with all red edges between them lacks both a red $C_{2k+1}$ and a blue $W_{2j}$.  Thus $4j+1$ is a clear lower bound for this Ramsey number.  The $k=1$ and $k=2$ cases of this conjecture were resolved in \cite{kone} and \cite{ktwo} respectively, but prior to 2015 very little headway had been made on this problem.

In 2015, Sanhueza-Matamala made significant progress toward settling this conjecture. 

\begin{theorem} \emph{\cite{san}} For $2<k<j$, $R(C_{2k+1},W_{2j}) \le 4j+334$.
\end{theorem}

In this paper, we prove a stronger result, which more or less settles the conjecture.

\begin{theorem} $R(C_{2k+1},W_{2j})=4j+1$ if $j \ge 212299$, $j-k \ge 251$, and $k<j<3k/2$.
\end{theorem}

Note that this proves the conjecture of Zhang, Zhang, and Chen in almost all cases.  In particular, if $j-k \ge 70766$ then $j>3(j-k) \ge 212298$, so $j \ge 212299$, $j-k \ge 251$, so Theorem 2 applies and therefore Conjecture 1 is true in this case.  Thus the only cases left unsolved have $j=k+O(1)$.  The conjecture is for $k<j<3k/2$, and so the cases left unsolved are a negligible subset of all cases.

In order to prove Theorem 2, we will first build up much of the machinery from \cite{san}.  The paper is organized as follows.  In the classical results section, we will state some classical theorems from the literature related to pancyclicity and circumference.  Then, in the preliminary lemmas section, we will find a vertex with large blue degree, and use this vertex to construct a graph $H$.  Assuming that there is no red $C_{2k+1}$ or blue $W_{2j}$, we will use then prove that $H^B$, the blue subgraph of $H$, has small circumference.  In the next section, we will then use one of the classical theorems, due to Fan \cite{fan}, as well as bounds on Ramsey numbers to find a large red cycle.  This yields enough structure to ultimately derive a contradiction, which we do in the final section.

\section{Classical Results}

There are a few results from graph theory which we will use throughout this paper.  The first of these results, due to Brandt et al., explains the presence of the number $251$ in Theorem 2, and in fact it also explains the presence of the number $334$ in Theorem $1$.

\begin{theorem}\emph{\cite{brandt1}} If all vertices of a non-bipartite $2$-connected graph on $n$ vertices have degree at least $\frac{n}{4}+250$, the graph is either weakly pancyclic or it has all cycle lengths between $4$ and $c(G)$ except length $5$. \end{theorem}

We use a slightly modified version of this result.

\begin{lemma}
If all vertices of a non-bipartite $2$-connected graph on $n$ vertices have degree at least $\frac{n}{4}+250$, the graph has all cycles lengths between $6$ and $c(G)$.
\end{lemma}

\begin{proof}
Assume this is not the case.  By Theorem $3$ we can assume the graph is weakly pancyclic.  Thus all cycles of the graph have length at least $7$, and the degree condition implies that $n \ge 4$.  Now consider $4$ vertices $v_1$, $v_2$, $v_3$, and $v_4$ of this graph.  If two of these vertices $v_i$ and $v_j$, have at least two common neighbors $a$ and $b$ then $v_iav_jb$ is a $4$-cycle, a contradiction.  Thus, for all $i<j$ we have $|N(v_i) \cap N(v_j)| \le 1$.  Then using the principle of inclusion-exclusion, $n \ge |\bigcup_{i=1}^{4} N(v_i)| \ge \sum_{i=1}^{4}|N(v_i)|-\sum_{1 \le i<j \le 4}|N(v_i) \cap N(v_j)| \ge 4(\frac{n}{4}+250)-6=n+994$, a contradiction.
\end{proof}

We also employ the next result, due to Fan \cite{fan}, a short proof of which can be found in \cite{feng}. This is the aforementioned result about the circumference, and it will be crucial later in the proof.

\begin{theorem}
\emph{\cite{fan}} Let $G$ be a $2$-connected graph on $n$ vertices.  Furthermore, suppose that for all pairs of nonadjacent vertices $v_1$ and $v_2$ which have a common neighbor, at least one of them has degree at least $\frac{c}{2}$.  Then $c(G) \ge \min \{c,n\}$.
\end{theorem}

We will use another pancyclicity result due to Brandt.

\begin{theorem}

\emph{\cite{brandt2}} Let $G$ be a non-bipartite graph on $n$ vertices with more than $\frac{(n-1)^2}{4}+1$ edges. Then $G$ is weakly pancyclic and contains a triangle.
\end{theorem}

We will also use yet another pancyclicity result due to Bondy.

\begin{theorem}

\emph{\cite{bondy}} Let $G$ be a graph on $n \ge 3$ vertices such that all vertices of $G$ have degree at least $\frac{n}{2}$.  Then either $G$ is pancyclic or $n$ is even and $G=K_{n/2, n/2}$ is a complete bipartite graph.

\end{theorem}

Finally, we employ a result about Ramsey numbers of pairs of even cycles due to Faudree and Schelp.

\begin{theorem} \emph{\cite{faudree}} If $b \ge a \ge 3$ then $R(C_{2a},C_{2b})=2b+a-1$. \end{theorem}

\section{The Main Proof}

\subsection{Preliminary Lemmas}

Throughout the remainder of the paper, we will argue by contradiction, assuming that some graph $G$ on $4j+1$ vertices is such that $G^R$ does not contain $C_{2k+1}$ and that $G^B$ does not contain $W_{2j}$, under the assumption $k+250<j<3k/2$ and $j \ge 212299$.  Note we may assume that $k>\frac{2j}{3}>141533$, so $k \ge 141534$.  So throughout the proof we will assume \[ k+250<j<3k/2, j \ge 212299, k \ge 141534 \tag{JK}.\] The general strategy for this section will be to emulate the methods of \cite{san}, and prove a number of lemmas about $G$ which are identical to or analogous to the corresponding lemmas of \cite{san}.  Later, we will use these results along with Theorem 5 to prove Theorem 2.

The first of these lemmas is about $2$-connectivity, a condition which appears prominently in Theorems 4 and 5 of the last section.  It is identical to Lemma 27 of \cite{san} but we reproduce a proof here for completeness.

\begin{lemma} If $\delta^R(G) \ge j$, then $G^R$ is $2$-connected.  \end{lemma}

\begin{proof} Assume the contrary, so there exists a subset $S$ of the vertices of $G$ with $|S| \le 1$ such that $(G-S)^R$ is disconnected.  Note that $\delta^R(G-S) \ge \delta^R(G)-|S| \ge \delta^R(G)-1 \ge j-1$.  It follows that every connected component of $(G-S)^R$ has size at least $j$.  

Clearly, any edge containing vertices in different components must be blue.  If $(G-S)^R$ has at least three blue components, then there is a blue cycle of length $2j$ involving two of these components.  Any vertex from a third component is incident to all these $2j$ vertices with blue edges, so $G$ has a blue $W_{2j}$.  This is a contradiction.  Thus $(G-S)^R$ has exactly two components.  These two components have at least $4j$ total vertices.  Denote the components as $C_1$ and $C_2$ so that $|C_1| \ge |C_2|$, and $|C_1| \ge 2j$.  Now, if $\delta^R(G[C_1])<|C_1|-j$ then there exists a vertex $v$ of $G[C_1]$ such that $j$ other vertices of $G[C_1]$ as well as $j$ vertices of $G[C_2]$ are in its blue neighborhood.  This then forms a blue $W_{2j}$, a contradiction.  Thus $\delta^R(G[C_1]) \ge |C_1|-j \ge |C_1|/2$.  By Theorem 7 $(G[C_1])^R$ is either pancyclic or complete bipartite with its parts of equal size.

Since $|C_1| \ge 2j \ge 2k+1$, if $(G[C_1])^R$ is pancyclic the graph contains a red $C_{2k+1}$, a contradiction.  Thus $(G[C_1])^R$ is complete bipartite with parts of equal size.  So $\delta^R(G[C_1])=|C_1|-j=|C_1|/2$.  Thus $|C_1|=2j$.  It follows $|C_2|=2j$ as well.  Now, pick a vertex $v$ of $G[C_1]$, which we try to use as the hub of a wheel.  It connects with blue edges to $j-1$ vertices of $G[C_1]$ and all $2j$ vertices of $G[C_2]$.  Now, if any vertex $v'$ of $G[C_2]$ has blue edges to vertices $x$ and $y$ of $G[C_2]$, then this forms a blue $W_{2j}$ with a hub at $v$.  Thus $\Delta^B(G[C_2]) \le 1$.  Hence $\delta^R(G[C_2]) \ge 2j-2>j$, so by Theorem 7 $(G[C_2])^R$ is pancyclic and contains a cycle of length $2k+1$, a contradiction. \end{proof}

\begin{lemma} (cf. Lemma 28 of \cite{san}) $G^R$ is not bipartite. \end{lemma}

\begin{proof} 

If this is not the case, because $|G|=4j+1$ one part of the partition has size at least $2j+1$.  Thus $G^B$ contains a complete graph on at least $2j+1$ vertices and has a blue $W_{2j}$, a contradiction.

\end{proof}

\begin{lemma} (cf. Lemma 29 of \cite{san}) $\Delta^B(G) \ge 3j-250$. \end{lemma}

\begin{proof} 

Assume for contradiction that all vertices of $G$ have blue degree at most $3j-251$.  Since $G$ has $4j+1$ vertices, we get that $\delta^R(G) \ge j+251$. So by Lemma 9, we have that $G^R$ is $2$-connected.  Now, we have that $\delta^R(G) \ge j+251>\frac{4j+1}{4}+250$.  By Lemma 10, $G^R$ is not bipartite.  Hence, by Lemma 4, $G^R$ contains all cycles of length between $6$ and $c(G^R)$.  Because $G^R$ is $2$-connected with all degrees at least $j+251$, certainly for any pair of nonadjacent vertices $v_1$ and $v_2$ with a common neighbor, at least one of them has degree $j+251$.  Then taking $c=2j+502$ in Theorem $5$, we have $c(G^R) \ge 2j+502>2k+1$.  Thus, $G^R$ has a red cycle of length $2k+1$, a contradiction. \end{proof}

As in \cite{san}, we take a vertex $v$ of maximum blue degree, and let $H$ be the graph induced by $N^B(v)$.  We have that \[ |H| \ge 3j-250\tag{1} \]

by Lemma 11.  For most of the rest of the proof, we will focus on $H$ and we will show it has either a red $C_{2k+1}$ or a blue $C_{2j}$.  Although $R(C_{2k+1},C_{2j})>3j-250$, we will prove a series of lemmas about $H$ and will ultimately show it has a red $C_{2k+1}$ or a blue $C_{2j}$. 

\begin{lemma} (cf. Lemma 30 of \cite{san}) $c(H^R)>2k+1$. \end{lemma}

\begin{proof}  Note that $R(C_{2k+2},C_{2j})=2j+k<3j-250$ due to Theorem 8, using $(JK)$.  As such, if $c(H^R) \le 2k+1$ then $H$ does not contain a red copy of $C_{2k+2}$, so it contains a blue copy of $C_{2j}$.  But then with $v$ this forms a blue copy of $W_{2j}$, a contradiction.   \end{proof}

\begin{lemma} (cf. Lemma 31 of \cite{san}) $H^B$ is not bipartite.  \end{lemma}

\begin{proof}

Assume for the sake of a contradiction that $H^B$ is bipartite.  If $H_1$ and $H_2$ are the two parts of the partition, both have size at most $2k$ since otherwise there would be a red copy of $C_{2k+1}$ in $H$.  Hence, both parts have size at least $|H|-2k \ge (3j-250)-(2j-502)=j+252$, using $(1)$.

If there were two disjoint red edges between $H_1$ and $H_2$ then using the fact that the edges among $H_1$ and $H_2$ were all red, $H$ would have all cycle lengths between $4$ and $|H|>2k+1$, so $H$ would have a red $C_{2k+1}$, which is a contradiction.  Thus there is a vertex $w \in V(H)$ such that $(H-w)^{B}$ is a complete bipartite graph with both parts of size at least $j+251$, and then $H^B$ contains $K_{j,j}$.  It follows that $H^B$ contains $C_{2j}$, and so with $v$ this forms a copy of $W_{2j}$ inside $G^B$, again a contradiction. \end{proof}

The next lemma is comparable to Lemma 32 of \cite{san}.  The proof is more complicated than in the case solved by Sanhueza-Matamala, because we assume that $|G|=4j+1$.

\begin{lemma} $H^R$ is not bipartite. \end{lemma}

\begin{proof} 

Assume $H^R$ is bipartite with parts $H_1$ and $H_2$.  If $H^B$ contains $C_{2j}$, then $G^B$ contains $W_{2j}$, so we must have that $|H_i| \le 2j-1$.  So using $(1)$, we have that $|H_i| \ge |H|-|H_{3-i}| \ge (3j-250)-(2j-1)=j-249$ as well.

There cannot be two disjoint blue edges between $H_1$ and $H_2$, because then $H^B$ would contain blue cycles of all lengths between $4$ and $|H| \ge 3j-250>2j$, again using $(1)$.

Thus, we can find a vertex $x \in V(H)$ so that $(H-x)^{R}$ is a complete bipartite graph.  Both parts $P_1$ and $P_2$ of $(H-x)^R$ are of size at least $j-250 \ge k+1$, since $P_i=H_i-x$ if $x \in H_i$ and $P_i=H_i$ otherwise for $i=1,2$.

Recalling that $H=N^B(v)$, let $Z=N^R(v)$.  We will now prove that at least one of the $P_i$ has only blue edges to $Z$.  Say there are red edges from some vertex $p_1$ of $P_1$ to $Z$ and from a vertex $p_2$ of $P_2$ to $Z$.  Then if there is some $z$ such that $zp_1$ and $zp_2$ are both red, because $|P_i| \ge k$ we use $p_1zp_2$ as well as a $K_{k,k}$ subgraph of $(H-x)^R$ containing both $p_1$ and $p_2$ to get a red $C_{2k+1}$.  If there is some $z_1$, $z_2$ so that $z_1p_1$ and $z_2p_2$ are both red then $p_1z_1vz_2p_2$ is a red path of length $4$.  Again this path can be completed to form a red cycle of length $2k+1$, this time using a $K_{k-1,k-1}$ subgraph of $(H-X)^R$ containing $p_1$ and $p_2$. It follows that at least one of the $P_i$ has only blue edges to $Z$ here.  Assume this is $P_1$.  We have \[|P_1| \ge |H_1|-1 \ge j-250\tag{P1}\] and \[|P_2| \le |H_2| \le 2j-1\tag{P2}\] from the bounds on $|H_i|$.  The vertices of $P_1$ have only blue edges with other elements of $P_1$, with $Z$, and with $v$. This motivates using a vertex of $P_1$ as the hub of a blue $W_{2j}$. 

If indeed $|P_1| \ge j+1$, pick some vertex $p$ of $P_1$ to use as a hub, and note that $P_1$, $\{v\}$, $Z$ in total have at least $(4j+1)-|P_2|-1 \ge 2j+1$ vertices, using the fact that $|P_2| \le |H_2| \le 2j-1$.  As such, we can use the at least $2j$ combined vertices of $P_1$, $\{v\}$, and $Z$ (other than $p$) to make a blue $C_{2j}$, since at least $j$ are in $P_1$ and thus have only blue edges between them while the others connect with only blue edges to $P_1 \setminus \{p\}$. Along with $p$, this forms a blue $W_{2j}$, a contradiction. 

Thus we may assume $|P_1| \le j$.  Using $(P2)$, we have that $|H| \le |P_1|+|P_2|+1=3j$.  This implies that $|Z|=|G|-|H|-|\{v\}| \ge j$.  Furthermore, we have that $|Z| \le 4j-|H| \le j+250$ by $(1)$.  Note also that $|P_1|+|P_2|+|\{v\}|+|Z| \ge 4j$, because the sets $P_1$, $P_2$, $\{v\}$, and $Z$ together contain all $4j+1$ vertices of $G$ except for $x$.  Using $(P2)$, we have $|P_1|+|Z|+|\{v\}| \ge 4j-|P_2| \ge 2j+1$.

Recall $(P1)$.  If there are $502$ vertex-disjoint blue edges among the vertices of $Z$, we use a vertex of $P_1$ as a hub, and using the at least $j-251$ remaining vertices in $P_1$, the $502$ edges among $Z$, and the at least $j-753$ left over vertices of $Z$, $P_1$, and $\{v\}$ (since $|Z|+|P_1|+|\{v\}| \ge 2j+1$) we obtain a blue cycle of length $2j$ and thus a copy of $W_{2j}$.

Else, there is a maximal matching $M$ in $Z^B$ with at most $501$ blue edges, and so at most $1002$ vertices of $Z$ are incident to edges of $M$.  Let $Z'$ be the set of vertices of $Z$ which are not incident to edges of $M$.  Hence $Z' \subset Z$ is such that $Z'^R$ is complete and $|Z'| \ge |Z|-1002 \ge j-1002$.  Now, if two vertices of $Z'$ have red edges to different vertices of $P_2$ we use those edges, the edge between the two vertices of $Z'$, and the edges between $P_1$ and $P_2$ to get a red copy of $C_{2k+1}$, and we are finished.  Recall that $2k+1<(j-250)+(j-249)+2=2j-497$ because $j-k>250$.

Otherwise, all edges between $Z'$ and $P_2$ are blue, or there is a red edge between some $a$ in $Z'$ and some $b$ in $P_2$.  In either case we can find $a \in Z'$ so that all edges between $Z' \setminus \{a\}$ and $P_2$ are blue.  We have $|P_2| \ge |H|-|P_1|-1 \ge 2j-251$.  Taking any vertex of $P_2$ as a hub, and using the at least $j-1003 \ge 252$ vertices of $Z' \setminus \{a\}$ as well as the at least $2j-252$ other vertices of $P_2$, we find a copy of $W_{2j}$, finishing the proof. \end{proof}

The next lemma is comparable to Lemma 33 of \cite{san}.

\begin{lemma}

$H^B$ is weakly pancyclic and contains a triangle.

\end{lemma}

\begin{proof}

By Lemmas 13 and 14, neither $H^R$ nor $H^B$ is bipartite, and one of them contains more than $(|H|-1)^2/4+1$ edges.  Hence, by Theorem $6$, one of $H^R$ and $H^B$ is weakly pancyclic and contains a triangle.  Since we have $c(H^R)>2k+1$, this cannot be $H^R$.

Hence, $H^B$ is weakly pancyclic and contains a triangle. \end{proof}

We prove another lemma about $H^B$, comparable to Lemma 34 of \cite{san}.

\begin{lemma}

$\delta^B(H) \ge j-250$.

\end{lemma}

\begin{proof}

Pick a vertex $w \in V(H)$.  Because $R(C_{2j-500},C_{2j})=3j-251$ by Theorem $8$, and $(H-w)^{B}$ does not have a copy of $C_{2j}$, we have that $(H-w)^{R}$ has a copy of $C_{2j-500}$.  Pick such a copy and call it $C'$.  Note $2j-500>2k>2k-1$.

Note that w cannot have red edges to two vertices in $V(C')$ at distance $2k-1$ along $C$, else there would be a red $C_{2k+1}$ in $G$. Hence, there is at least one blue edge between $w$ and $V(C')$ for each red edge between $w$ and $V(C')$. Therefore, at least half of the edges between $w$ and $V(C')$ are blue.  Hence the blue degree of $w$ in $H$ is at least $j-250$, as desired. \end{proof}

The next lemma about $H^B$ is comparable to a result proved in the conclusion of \cite{san}.

\begin{lemma}

$H^B$ is $2$-connected.

\end{lemma}

\begin{proof}

Assume the contrary, and choose a subset $S \subset V(H)$ of at most $1$ vertex so that $H^B-S$ is disconnected.  Since $H^B-S$ has all degrees of size at least $j-251$ by Lemma $16$, all components are of size at least $j-250 \ge k+1$.  Because edges between different components are red, if there are at least $3$ components we can find a red copy of $C_{2k+1}$.  Hence, $H-S$ has two connected components $C_1$ and $C_2$ of sizes at least $k+1$.  

If $C_1$ or $C_2$ contains any red edge, then the graph contains $C_{2k+1}$, since $K_{k+1,k+1}$ contains $C_{2k+1}$ after adding any edge.  Thus, $C_1$ and $C_2$ are complete blue graphs.  Since $H^R$ is not bipartite by Lemma $14$, $S$ is nonempty and $H^B-S \neq H^B$.  As such, there exists $s \in V(H)$ such that $S=\{s\}$ which is in neither $C_1$ nor $C_2$.  Because $(C_1 \cup \{s\}, C_2)$ is not a bipartition of $H^R$, $s$ has a red edge to some vertex $v_1$ of $C_1$.  Similarly, because $(C_1, C_2 \cup \{s\})$ is not a bipartition of $H^R$, $s$ has a red edge to a vertex $v_2$ of $C_2$.  Using the edges $v_1sv_2$ and the fact that both $C_1$ and $C_2$ have at least $k+1$ elements, again we have a red copy of $C_{2k+1}$ in $H$, a contradiction. \end{proof}

The final lemma in this subsection is a quick result on the circumference of $H^B$.

\begin{lemma} $c(H^B)<2j$. \end{lemma}

\begin{proof} Assume the contrary, so $c(H^B) \ge 2j$.  By Lemma 15, $H^B$ is weakly pancyclic and contains a triangle.  Thus $H^B$ has a cycle of length exactly $2j$.  This cycle and $v$ together form a blue $C_{2j}$, a contradiction.  \end{proof}

\subsection{The Structure of a Large Cycle in the Graph}

In this subsection, given a graph $G$ with $4j+1$ vertices, we again we take a vertex $v$ of maximum blue degree, and let $H$ be the graph induced by $N^B(v)$.  The strategy is to use Theorem $5$ to find a large red cycle in $H$, and to use this cycle to develop more structure in $H$.

First, by Lemma 17, $H^B$ is $2$-connected.  By Lemma 18, $c(H^B)<2j$.  Since $|H| \ge 2j$, $c(H^B)<\min(2j,|H|)$.  By Theorem 5 applied to $H^B$, there are some two vertices of $H$, $a$ and $b$, both with blue degree less than $j$, such that the edge between them is red.

Recall that \[|H-a-b| \ge |H|-2 \ge 3j-252 \tag{HAB}\] by $(1)$.  By Theorem $8$, $R(C_{2j-502},C_{2j})=3j-252$.  Since $H^B$ contains no $C_{2j}$, it follows that $H-a-b$ contains a red $C_{2j-502}$, which we call $C$, so that \[|C|=2j-502\tag{C}.\]  Note $2j-502 \ge 2k$ by $(JK)$.  Label the vertices in $V(C)$ by calling some vertex $0$, the vertex immediately counterclockwise of it $1$, and so on.  Thus every vertex in $V(C)$ has a unique label between $0$ and $2j-503$ inclusive.  Taking labels modulo $|C|=2j-502$, for all $i$ the vertices $i$ and $i+1$ are incident in $C$.  Using $(HAB)$ and $(C)$, there are at least $|H-a-b|-|C| \ge 3j-252-(2j-502)=j+250$ vertices of $H$ that are not in $C$, and are not $a$ or $b$.  

If $a$ has a red edge to the vertex $x$ as well as to the vertex $x+(2k-1)$, we have a red copy of $C_{2k+1}$ using those edges from $a$ to $x$ and $x+(2k-1)$, and the red path of length $2k-1$ between $x$ and $x+(2k-1)$.  Hence, there is at least one blue edge between $a$ and $C$ for each red edge between $a$ and $C$.  It follows that at least half the edges from $a$ to $V(C)$ must be blue, so $a$ has at least $j-251$ blue edges to $V(C)$ and at most $j-1$ blue edges in $H$ in total.  In particular, of the $2j-502$ edges from $a$ to $C$, between $j-251$ and $j-1$ of them are blue whereas between $j-501$ and $j-251$ of them are red.  Analogous results hold for $b$.  

Note that $a$ and $b$ each have at least $j-251$ blue edges to $C$, but have a blue degree at most $j-1$ in $H^B$.  Therefore each of $a$ and $b$ has at most $250$ edges to the rest of $H$, which has cardinality $|H-a-b|-|C| \ge 3j-252-(2j-502)=j+250$ by $(HAB)$ and $(C)$.  Thus there is a subset $H'$ of $V(H)$ of size \[ |H'| \ge j+250-250-250=j-250\tag{2}\] such that all edges from $a$ and $b$ to $H'$ are red, $H'$ is disjoint from $C$, and $H'$ does not contain $a$ or $b$.  Fix such a choice of $H'$ and let $c$ be a vertex of $H'$.

We let $A_i$ be $1$ if the edge between $a$ and $i \in V(C)$ is red, and $0$ if it is blue.  Similarly, we let $B_i$ be $1$ if the edge between $b$ and $i \in V(C)$ is red, and $0$ if it is blue.  Next, we prove a series of lemmas about the edges between the vertices $a$ and $b$ on one hand, and the cycle $C$ on the other.  Developing this structure will be crucial for finishing the proof.  We will develop this structure through a series of lemmas, which will all be more or less of the same form.

\begin{lemma} Let $S_{a,19}=\{x: A_x=A_{x+2k-1}\}$ and let $S_{b,19}=\{x: B_x=B_{x+2k-1}\}$.   Then $|S_{a,19}| \le 500$, and similarly $|S_{b,19}| \le 500$.  \end{lemma}

\begin{proof} 

It suffices to prove the first inequality, because the second inequality is analogous.

If the edges from $a$ to $x$ and $x+(2k-1)$ are both red, then they form a red cycle of length $2k+1$ along with a path of length $2k-1$ on $C$.  Let $M$ be the map sending the vertex $x$ in $V(C)$ to the vertex $x+(2k-1)$ in $V(C)$.  Any of the at least $j-501$ vertices in $V(C)$ with a red edge to $a$ is sent to a vertex with a blue edge to $a$.  Because $M$ is bijective, there are at most $(j-1)-(j-501)=500$ other vertices in $V(C)$ that can be sent to vertices in $V(C)$ with a blue edge to $a$.  Hence, there are at most $500$ vertices $x \in V(C)$ with a blue edge to $a$ such that $x+(2k-1) \in V(C)$ also has a blue edge to $a$, establishing the lemma. \end{proof}

\begin{lemma}  We let $S_{a,20}=\{x: A_x=B_{x+2k-2}\}$ and $S_{b,20}=\{x: B_x=A_{x+2k-2}\}$.  Then $|S_{a,20}| \le 500$, and similarly $|S_{b,20}| \le 500$.  \end{lemma}

\begin{proof} 

As in Lemma 19, it suffices to prove the first inequality.

If the edges from $a$ to $x$ and $b$ to $x+(2k-2)$ are both red, then these form a red cycle of length $2k+1$ along with the edge from $a$ to $b$ and the path from $x$ to $x+(2k-2)$ along $C$.  Let $M$ be the map sending the vertex $x$ in $V(C)$ to the vertex $x+(2k-2)$ in $V(C)$.  Each of the at least $j-501$ vertices in $V(C)$ with a red edge to $a$ is sent to a vertex with a blue edge to $b$.  Because $M$ is bijective, there are at most $(j-1)-(j-501)=500$ other vertices in $V(C)$ that can be sent to vertices in $V(C)$ with a blue edge to $a$.  Hence, there are at most $500$ vertices $x \in V(C)$ with a blue edge to $a$ such that $x+(2k-2) \in V(C)$ also has a blue edge to $a$, establishing the lemma. \end{proof}

\begin{lemma} Let $S_{a,21}=\{x: A_x=B_{x+2k-3}\}$ and $S_{b,21}=\{x: B_x=A_{x+2k-3}\}$.  Then  $|S_{a,21}| \le 500$, and similarly $|S_{b,21}| \le 500$.  \end{lemma}

\begin{proof} 

As before, it suffices to prove the first inequality.

If the edges from $a$ to $x$ and $b$ to $x+(2k-3)$ are both red, then these form a red cycle of length $2k+1$ along with the edge from $a$ to $c$, the edge from $c$ to $b$, and the path from $x$ to $x+(2k-3)$ along $C$.  Let $M$ be the map sending the vertex $x$ in $V(C)$ to the vertex $x+(2k-3)$ in $V(C)$.  Any of the at least $j-501$ vertices in $V(C)$ with a red edge to $a$ is sent to a vertex with a blue edge to $b$.  Because $M$ is bijective, there are at most $(j-1)-(j-501)=500$ other vertices in $V(C)$ that can be sent to vertices in $V(C)$ with a blue edge to $a$.  Hence, there are at most $500$ vertices $x \in V(C)$ with a blue edge to $a$ such that $x+(2k-3) \in V(C)$ also has a blue edge to $a$, establishing the lemma. \end{proof}

\begin{lemma} Let $S_{a,22}=\{x: A_x \neq B_{x-1} \}$ and $S_{b,22}=\{x: B_x \neq A_{x-1} \}$.  Then $|S_{a,22}| \le 1000$, and similarly $|S_{b,22}| \le 1000$. \end{lemma}

\begin{proof} Again, it suffices to prove the first inequality.  By Lemma 20, $A_x=1-B_{x+2k-2}$ for all $x$ not in $S_{a,20}$.  By Lemma 19, we have that $B_{x-1}=1-B_{(x-1)+(2k-1)}=1-B_{x+2k-2}$ whenever $x-1$ is not in $S_{b,19}$.  Thus $A_x=1-B_{x+2k-2}=B_{x-1}$ whenever $x$ is not in $S_{a,20}$ and $x-1$ is not in $S_{b,19}$.  Each of these happens for at most $500$ values of $x$.  Thus $A_x=1-B_{x+2k-2}=B_{x-1}$ for all but at most $1000$ values of $x$, completing the proof. \end{proof}

\begin{lemma} Let $S_{a,23}=\{x: A_x \neq A_{x+1}\}$ and $S_{b,23}=\{x: B_x \neq B_{x+1}\}$.  Then $|S_{a,23}| \le 1000$, and similarly $|S_{b,23}| \le 1000$. \end{lemma}

\begin{proof} As before, it suffices to prove the first inequality.

By Lemma 21, $A_{x+1}=1-B_{x+2k-2}$ whenever $x+1$ is not in $S_{a,21}$, which represents at most $500$ possible choices of $x$.  By Lemma 20, $A_x=1-B_{x+2k-2}$ whenever $x$ is not in $S_{a,20}$, which represents another at most $500$ choices of $x$.  Thus it follows $A_x=1-B_{x+2k-2}=A_{x+1}$ for all but at most $1000$ values of $x$, completing the proof.  \end{proof}

\begin{lemma} Let $S_{24}=\{x: A_x \neq B_x\}$.  Then $|S_{24}| \le 2000$.

\end{lemma}

\begin{proof} 

By Lemma 23, $A_x=A_{x+1}$ for all $x$ not in $S_{a,23}$.  By Lemma 22, $A_{x+1}=B_x$ whenever $x+1$ is not in $S_{a,22}$.  Hence $A_x=A_{x+1}=B_x$ whenever $x$ is not in $S_{a,23}$ and $x+1$ is not in $S_{a,22}$.  Each happens at most $1000$ times, so this equation holds for all but at most $2000$ values of $x$.  This establishes the lemma. \end{proof}

From Lemma 24, it follows that there are at least $(j-251)-2000=j-2251$ vertices in $V(C)$ that connect to both $a$ and $b$ with blue edges.  Let $\mathcal{B}_1$ be the set of such vertices.  Similarly, at least $(j-501)-2000=j-2501$ vertices in $V(C)$ connect to both $a$ and $b$ with a red edge.  Let $\mathcal{R}_1$ be the set of such vertices.  In other words, for all $f \in \mathcal{B}_1$, $A_f=B_f=0$ whereas for all $f \in \mathcal{R}_1$, $A_f=B_f=1$.

Now, note that if $j-k=251$, then by $(C)$ we have $|C|=k+1$.  If $x$ is not in $S_{a,19}$ it follows from Lemma 19 that $A_x=1-A_{x+2k-1}=1-A_{x-1}$, and if $x-1$ is not in $S_{a,23}$ it follows from Lemma 23 that $A_x=A_{x-1}$.  So for all $x$ either $x \in S_{a,19}$ or $x-1 \in S_{a,23}$.  But $x \in S_{a,19}$ for at most $500$ vertices and $x-1 \in S_{a,23}$ is the case for at most $1000$ vertices, so there are at most $1500$ vertices in $V(C)$.  Then $2j-502 \le 1500$, and $j \le 1001$, which contradicts the assumption in Theorem 2.  For the remainder of the proof, we may assume $j-k \ge 252$.

\subsection{Conclusion of the Proof of Theorem 2}

Recall that we have a set $H' \subset V(H)$ with $|H'| \ge j-250$ and such that vertices of $H'$ have only red edges to $a$ and $b$.  Furthermore, recall that we have a large red cycle $C$.  It is these structures that will play a primarily role in the final part of the proof.

Now, say there is a red edge between some vertex $c$ of $H'$ and some $x \in \mathcal{B}_1$.  Then the edge between $a$ and $x+(2k-2)$ is blue, else there would be a red $C_{2k+1}$ formed by the path from $x$ along $C$ to $x+(2k-2)$ to $a$ to $c$ and back to $x$.  Hence $A_{x+2k-2}=0$.  But $B_x=0$ also by the definition of $\mathcal{B}_1$. Hence, $x \in S_{b,20}$.  In particular, vertices of $\mathcal{B}_1$ not in $S_{b,20}$ lack red edges to $H'$.  Thus we can find $\mathcal{B}_2=\mathcal{B}_1 \setminus S_{b,20}$ with $|\mathcal{B}_2| \ge |\mathcal{B}_1|-|S_{b,20}| \ge j-2751$ such that every edge between a vertex of $\mathcal{B}_2$ and a vertex of $H'$ is blue.

The upshot of this is that in $H^B$, there is now a copy of the complete bipartite graph $K_{j-2751,j-250}$ formed by $\mathcal{B}_2$ on the one hand and by $H'$ on the other.  Recall that in order to exhibit a contradiction all we must do is show that $c(H^B) \ge 2j$, because this contradicts Lemma 18.  The idea now will be to find more vertices of $V(C)$ with blue edges to pairs of vertices of $\mathcal{B}_2$, and use these to complete a blue cycle of length at least $2j$ inside $H$. 

There are at least $j-2501$ vertices in $\mathcal{R}_1$ which connect to both $a$ and $b$ with a red edge.  Let $E$ be the set of $x \in V(C)$ so that at least one of the following conditions holds: $x+2k \in S_{b,22}, x+(2k-1) \in S_{a,22}, x-(2k-1) \in S_{a,22}$, $x-2k \in S_{a,23}$, $x-2k \in S_{b,20}$, or $x+2k \in S_{b,20}$.  Thus $|E| \le 4(1000)+2(500)=5000$.  Let $x \in \mathcal{R}_1 \setminus E$. Then $A_{x+(2k-1)}=0$ to prevent a red $C_{2k+1}$, and by Lemma 22, $B_{x+2k}=0$.  By Lemma 23, $A_{x+2k}=0$.  Similarly, $A_{x-(2k-1)}=0$, and by Lemma 22, $B_{x-2k}=0$.  By Lemma 23, $A_{x-2k}=0$.  Thus $x-2k$ and $x+2k$ are both in $\mathcal{B}_1$.  As neither $x+2k$ nor $x-2k$ is in $S_{b,20}$, both $x+2k$ and $x-2k$ are in $\mathcal{B}_2$.

Thus, there are at least $|\mathcal{R}_1|-|E| \ge |\mathcal{R}_1|-5000 \ge j-2501-5000=j-7501$ vertices $x \in V(C)$ such that $x+2k$ and $x-2k$ are both in $\mathcal{B}_2$ but $x \in \mathcal{R}_1$.  Let $\mathcal{R}_2$ be the set of such $x$, so $\mathcal{R}_2 \subset \mathcal{R}_1$ is a subset of $V(C)$.  Recall also that for any $x \in V(C)$ the edge between $x$ and $x-2k$ and the edge between $x$ and $x+2k$ are both blue to prevent the existence of a red $C_{2k+1}$. The vertices $x-2k$, $x$, and $x+2k$ are all distinct because $2k<2j-502<4k$.  

We have $|\mathcal{R}_2| \ge j-7501$.  Now, we will pick a set $L$ of $3001$ vertices of $\mathcal{R}_2$ such that no two are at distance $4k$ in $C$.  We do this through a greedy algorithm.  At any point, if at most $3000$ vertices of $\mathcal{R}_2$ have been picked, those vertices and the vertices of distance $4k$ away represent at most $3 \times 3000$ vertices of $\mathcal{R}_2$.  Thus there is another vertex of $\mathcal{R}_2$ that can be picked, because $j \ge 16502$ and so $|\mathcal{R}_2| \ge j-7501 > 3 \times 3000$.

These $3001$ chosen vertices $L$ of $\mathcal{R}_2$ all connect with blue edges to distinct pairs of vertices of $\mathcal{B}_2$.  The idea here is that we will use them to extend the complete blue bipartite graph between $H'$ and $\mathcal{B}_2$, and ultimately to find a blue $C_{2j}$ in $H$. It turns out that if $H'$ has a blue matching with at least $5502$ edges, we can do this immediately.

\begin{lemma} The maximal matching of $H'^B$ has at most $5501$ edges.

\end{lemma}

\begin{proof} Assume there is a matching of $H'^B$ with at least $5502$ edges.  Then there exists $5502$ disjoint blue edges of $G[H']$, say $a_ib_i$ for $1 \le i \le 5502$.  Label the vertices of $L$ as $\ell_i$ for $1 \le i \le 3001$.  For each $\ell_i$, let $c_i$ be one of the vertices of $V(C)$ at distance $2k$ from it, and let $d_i$ be the other such vertex.  Let $G_1$ be the subgraph of $H^B$ induced on the union of the vertex sets of $L$, $H'$, and $\mathcal{B}_2$.  Let $G_2$ be the graph minor of $G_1$ formed by contracting every edge $a_ib_i$ to a single vertex $x_i$, and contracting every path $c_i\ell_id_i$ to a single vertex $y_i$.  There are at least $j-250-5502=j-5752$ vertices of $H'$ and $x_i$ in $V(G_2)$.  We pick $j-5752$ of these, including all the $x_i$, and call them \emph{left vertices}.  There are also at least $j-2751-3001=j-5752$ vertices of $\mathcal{B}_2$ and $y_i$ in $V(G_2)$.  We pick $j-5752$ of these, including the $y_i$, and call them \emph{right vertices}.  All $x_i$ are formed through contraction of vertices of $H'$, and all $y_i$ through contraction of vertices of $\mathcal{B}_2$ and $L$.  Since there is an edge between every vertex of $H'$ and every vertex of $\mathcal{B}_2$, there is an edge between every left vertex and every right vertex.  Thus there is a cycle of length $(j-5752)+(j-5752)=2j-11504$ consisting of left vertices and right vertices.  Orient this cycle.  Now we will expand this cycle to a cycle of length $2j$ in $G_1$.  In general $f^{-}$ denotes the vertex before $f$ on $C$, and $f^{+}$ is the vertex after.  First, expand the $x_i$ in sequence, starting with $i=1$.  For $1 \le i \le 5502$, on the $i$th step, replace the edge directed toward $x_i$ from the vertex $x_i^{-}$ and the edge from $x_i$ to some $x_i^{+}$ with the edges from $x_i^{-}$ to $a_i$, $a_i$ to $b_i$, and $b_i$ to $x_i^{+}$.  After this, for $1 \le i \le 3001$, on the $5502+i$th step, replace the edges from $y_i^{-}$ to $y_i$ and $y_i$ to $y_i^{+}$ with edges from $y_i^{-}$ to $c_i$, $c_i$ to $\ell_i$, $\ell_i$ to $d_i$, and $d_i$ to $y_i^{+}$.  After all this is done, the result is a cycle on $(j-5752)+(j-5752)+5502 \times 1+3001 \times 2=2j$ vertices which is inside $G_1$.  Thus there is a $C_{2j}$ in $H^B$, and a $W_{2j}$ in $G^B$, a contradiction. \end{proof}

Hence, we may assume that there exists maximal matching $M$ of $H'^B$ with at most $5501$ edges.  Let $J=H' \setminus V(M)$.  We have \[ |J| \ge |H'|-2(5501)=|H'|-11002 \ge j-250-11002=j-11252\tag{3} \] by $(2)$, and $J^R$ is complete by the maximality of the matching.

From here on out, the idea is to find a lot of blue edges between $J$ and $C$, using the fact that $H$ lacks a red $C_{2k+1}$.  Using these blue edges, we will find a $C_{2j}$ in $H$.   In order to find such a $C_{2j}$, we need another lemma.

\begin{lemma} If $\frac{j}{3} \ge T \ge 73$, and some set $L'$ of at least $j-T$ vertices of $H$ all connect with blue edges to some set $R'$ of at least $j+4T-252$ vertices in $V(C)$, $H^B$ has a $C_{2j}$.  \end{lemma}

\begin{proof}
First, we claim the subgraph of $H^B$ induced on $R'$ has a matching with $2T$ edges.  Consider a maximal matching $M$.  If $M$ has at most $2T-1$ edges, then there are at least $(j+4T-252)-(4T-2)=j-250$ vertices of $R'$ that are not in $V(M)$.  Hence, more than half of the elements of $V(C)$ are not in $V(M)$.  Thus there exists $x \in V(C)$ so that neither $x$ nor $x+2k$ is in $V(M)$.  But the edge linking $x$ and $x+2k$ is blue, to avoid forming a red $C_{2k+1}$.  This means we can adjoin it to $M$, and this contradicts the maximality of $M$.

Now, note that $j+4T-252 \ge j+T$ because $T \ge 73$.  Call the matching $a_ib_i$, where $1 \le i \le 2T$. Label $2T$ vertices of $L'$ as $c_i$ for $1 \le i \le 2T$.  Consider the path starting at $c_1$ and going from $c_i$ to $a_i$ to $b_i$ and back to $c_{i+1}$ for $1 \le i \le 2T-1$, and then going from $c_{2T}$ to $a_{2T}$ to $b_{2T}$.  After tracing this path there are $j-3T \ge 0$ unused vertices in $L'$ and at least $j-3T \ge 0$ unused vertices in $R'$, so using the fact that there is a blue edge between any vertex of $L'$ and any vertex of $R'$ we form a blue cycle of length $2j$.  This finishes the proof. \end{proof}

Now, if all edges from $J$ to $C$ are blue, then $j \ge 3 \times 11252=33756$, and furthermore $|C|=2j-502 \ge j+4 \times 11252-252$, because $j \ge 45258$.  Next, let $T=11252$, let $L'=J$, which is of the appropriate size by $(3)$, and let $R'=C$.  From Lemma 26 it follows that $H^B$ contains a $C_{2j}$ and we are done.

Else, there exists a vertex $g \in J$ such that there is a red edge between $g$ and a vertex $x_0$ of $V(C)$.  Then let $J_1=J \setminus \{g\}$, so $|J_1| \ge |J|-1 \ge j-11253$ by $(3)$.  For any vertex $g'$ of $J_1$, there exists red paths from $g$ to $g'$ of all lengths between $1$ and $j-11253$ inclusive.  Hence for any $g'$ of $J_1$, the edge between $g'$ and any vertex in $V(C)$ with a path of length $2k-2$ and $(2k-1)-(j-11253)=2k-j+11252$ along $C$ to $x_0$ is blue, to avoid a red $C_{2k+1}$.

We are now nearly finished.  We have already forced a huge number of blue edges between $J_1$ and $C$.  Recall that $11251<2k-j+11251<2k-2<2j-502$. Starting at any fixed vertex, there are a total of $j-11252$ vertices one can reach by going clockwise around $C$ for between $2k-j+11251$ and $2k-2$ vertices.  

So we have that some set $J_1$ of $j-11253$ vertices of $J$ are incident with blue edges to some consecutive set of $j-11252$ vertices in $V(C)$.  In particular the vertices of $V(C)$ obtained by starting at $x_0$ and going clockwise around $C$ for between $2k-j+11251$ and $2k-2$ vertices.  Assume without loss of generality these $j-11252$ vertices are numbered $1, \ldots, j-11252$ with $2$ counterclockwise of $1$.  Assume that there are no red edges between elements of $V(C)$ with labels between $1$  and $j+4(11253)-252=j+44760 \le 2j-502$ and vertices of $J_1$.  Then set $T=11253$, $L'=J_1$, and let $R'$ be the vertices with labels at most $j+44760$.  Since $T \le \frac{j}{3}$, by Lemma 26, $H^B$ contains a $C_{2j}$, so $G^B$ contains a blue $W_{2j}$ and we are done.  Hence, there exists some vertex $x_1$ in $V(C)$ with $j-11251 \le x_1 \le j+44759$ and a vertex $u$ of $J_1$ so that there is a red edge between $x_1$ and $u$.

We proceed similarly as before.  All vertices of $V(C)$ that can be reached from $x_1$ by a path of length between $2k-j+11252$ and $2k-2$ must have blue edges to all vertices of $J \setminus \{u\}=J_2$.  Of course $|J_2| \ge |J|-1 \ge j-11253$ by $(3)$.  Let $B$ represent the set of vertices reachable from $x_1$ through moving clockwise a distance between $2k-j+11252$ and $2k-2$ along $C$, and $D$ represent the set of vertices reachable from $x_1$ through moving clockwise a distance between $2k-j+11252$ and $2k-2$ along $C$.  We have that $B$ and $D$ are sets of consecutive vertices in $V(C)$, and $11251<2k-j+11251<2k-2<2j-502$ so neither $B$ and $D$ contains $x_1$.  If $B$ and $D$ do not intersect, let $R'=B \cup D$, $L'=J_2$, and $T=11253$.  Then $|R'| \ge 2(j-11253) \ge j+4 \times 11253-252$ because $j \ge 67266$.  So by Lemma 26, $H^B$ contains $C_{2j}$ and we are done.  Else $B$ and $D$ intersect.  Because the vertices of $B$ and $D$ are consecutive, do not contain $x_1$, and are symmetric about $x_1+(j-251)$, they both contain $x_1+(j-251)$.  In fact, it follows that $B \cup D$ is symmetric about $x_1+(j-251)$ and is itself a consecutive set of vertices in $C$.  By symmetry $|B \cup D|$ must be odd.

Let $J_3=J \setminus \{u,g\}$, so $|J_3| \ge |J|-2 \ge j-11254$ by $(3)$.  All edges between $B \cup D$ and $J_3$ are blue.  Thus a set of at least $j-11253$ vertices centered somewhere between $(j-11251)-(j-251)=-11000$ and $(j+44759)-(j-251)=44508$ has blue edges to all of $J_3$, as do the vertices $1, \ldots, j-11252$.

Since $j \ge 212299$ and $j-11253 \ge 201046$, we have all the vertices with labels between $44508-100523=-56015$ and $j-11252$ are blue to $J_3$.  This is at least $j+44764$ vertices.

Because $|J_3| \ge j-11254$, we can let $T=11254$, let $L'=J_3$, and let $R'$ be the at least $j+44764=j+4 \times 11254-252$ vertices of $V(C)$ with blue edges to everything in $J_3=L'$.  By Lemma 26, it follows that $H^B$ contains $C_{2j}$, so $G$ contains a blue copy of $W_{2j}$ as desired.

Our proof of Theorem 2 is complete.

\section{Acknowledgements}

This research was made possible by NSF grants NSF-1358659 and NSA H98230-16-1-0026.  It was conducted at the Duluth REU, run by Joe Gallian.  Many thanks to Matthew Brennan for suggesting the problem.  Also, thanks to Matthew Brennan and Joe Gallian for helpful comments.

\end{document}